\documentclass[12pt,letterpaper]{amsart}

\usepackage{ae,aecompl}
\usepackage[T1]{fontenc}
\usepackage[latin1]{inputenc}
\usepackage{hyperref}
\usepackage{amsrefs}
\usepackage{xspace}
\usepackage{enumerate}
\usepackage{booktabs}

\newtheorem{theorem}{Theorem}[section]
\newtheorem{proposition}[theorem]{Proposition}
\newtheorem{lemma}[theorem]{Lemma}

\theoremstyle{definition}
\newtheorem{definition}[theorem]{Definition}

\theoremstyle{remark}
\newtheorem{remark}[theorem]{Remark}

\numberwithin{equation}{section}

\newcommand{\asc}[1]{\ensuremath{\mathsf{#1}}}
\newcommand{\R}{\ensuremath{\mathbb{R}}}
\newcommand{\lutzm}[4]{\asc{{}^{#1}{#2}^{#3}_{#4}}}
\newcommand{\surff}{\footnotesize}
\newcommand{\hem}{\hspace{2em}}

\DeclareMathOperator{\conv}{conv}
\DeclareMathOperator{\str}{star}

\begin{document}

\title[Non-Realizable Minimal Vertex Triangulations of Surfaces]{Non-Realizable Minimal Vertex Triangulations of Surfaces:\\
  Showing Non-Realizability using Oriented Matroids and Satisfiability
Solvers}
\author{Lars Schewe}
\address{Lars Schewe\\Fachbereich Mathematik, AG Optimierung\\
  TU Darmstadt\\Schlossgartenstraße 7\\64289 Darmstadt\\Germany}
\email{schewe@mathematik.tu-darmstadt.de}
\thanks{The author was supported by a scholarship of the Deutsche
  Telekom Foundation}


\date{}

\begin{abstract}
  We show that no minimal vertex triangulation of a closed, connected,
  orientable 2-manifold of genus $6$ admits a polyhedral embedding in
  $\R^3$. We also provide examples of minimal vertex triangulations of
  closed, connected, orientable 2-manifolds of genus $5$ that do not
  admit any polyhedral embeddings. We construct a new infinite family
  of non-realizable triangulations of surfaces. These results were
  achieved by transforming the problem of finding suitable oriented
  matroids into a satisfiability problem. This method can be applied
  to other geometric realizability problems, e.g. for face lattices of
  polytopes.
\end{abstract}

\maketitle

Grünbaum conjectured \cite{MR1976856}*{Exercise 13.2.3} that all
triangulated surfaces (compact, orientable, connected, 2-dimensional
manifolds without boundary) admit polyhedral embeddings in $\R^3$.
This conjecture was shown to be false by Bokowski and Guedes de
Oliveira \cite{bokowski.guedesdeoliveira:generation}.  They showed
that one special triangulation with 12 vertices of a surface of genus
6 does not admit a polyhedral embedding in $\R^3$. Recently,
Archdeacon et al. \cite{MR2352708} settled the case of genus $1$ by
showing that all triangulations of the torus admit a polyhedral
embedding.

Still, triangulated surfaces with polyhedral embeddings can be quite
complicated. McMullen, Schulz, and Wills constructed polyhedral
embeddings of triangulated surfaces with $n$ vertices of genus
$\Theta(n\log n)$ (\cite{MR727027}, see also \cite{ziegler:surfaces}).
However, a gap remains: Jungerman and Ringel
\cites{MR586595,MR0349461} showed that $n$ vertices suffice to
triangulate a surface of genus $\Theta(n^2)$ and explicitly
constructed such triangulations.

So, can we construct polyhedral embeddings of triangulated surfaces
with few vertices? In the case of $2$-spheres the combinatorial bound
is sharp; this is a consequence of Steinitz's Theorem~\cite{Steinitz}.
It is known that all vertex minimal triangulations of surfaces up to
genus 4 admit polyhedral embeddings (genus 1 was first done by
Cs{\'a}sz{\'a}r \cite{MR0035029}, the cases of genus 2 and 3 were
solved by Lutz and Bokowski \cite{math.CO/0506316}, Lutz
\cite{math.CO/0506316} and Hougardy, Lutz, and Zelke
\cite{hougardy.ea:intersection}).

Our main result is that none of the vertex minimal triangulations of a
surface of genus 6 admits a realization in $\R^3$.  Moreover, three
minimal triangulations of a surface of genus 5 do not admit
realizations either. A small modification of one of triangulations
help us to construct a new infinite class of non-realizable
triangulated surfaces. For all results we use an improved method to
construct oriented matroids that are admissible for the surface in
question. The method can also be applied to embedding problems for
general simplical complexes in arbitrary dimensions. A small
modification of the method allows us to also treat immersions of
simplical complexes. Using this modification we can rule out for all
but one triangulation of the surface of genus $6$ with $12$ vertices
that it can be immersed into $\R^3$.

The new method we propose to generate oriented matroids reduces the
generation problem to an instance of the satisfiability problem. This
allows us to use well-tuned software and speeds up the checking
process immensly. As oriented matroids have been used to tackle other
geometric realizability problems, our method gives more effective
algorithms for these problems as well.

\section{Results}
\label{sec:results}

Using the algorithm given below, it was possible to show the following
theorems:  

\begin{theorem}
  No triangulation of a surface of genus 6 with 12 vertices admits a
  polyhedral realization in $\R^3$.
\end{theorem}
 
The theorem is a consequence of the following proposition. A key step
is the classification of combinatorial surfaces with $12$ vertices of
genus $6$ by Altshuler, Bokowski and Schuchert
\cite{altshuler.ea:neighborly}.

\begin{proposition}
  None of the $59$ combinatorial surfaces with $12$ vertices of genus
  $6$ admits an acyclic, uniform oriented matroid.
\end{proposition}

The situation is more difficult in the case of genus $5$. To
triangulate a surface of genus $5$ we also need at least $12$
vertices. However, there are far more possibilities ($751\,593$ as
enumerated by Lutz and Sulanke \cite{lutz.sulanke:isomorphism}) than
in the case of genus $6$.

\begin{table}
  \caption{Number of combinatorial triangulations}
  \centering
  \begin{tabular}{rrr
    }
    \toprule
    $g$ & $n_{\min}$ & \# 
    \\
    \midrule
    0 & 4 & 1 
    \\
    1 & 7 & 1 
    \\
    2 & 10 & 865 
    \\
    3 & 10 & 20  
    \\
    4 & 11 & 821 
    \\
    5 & 12 & 751\,593 
    \\
    6 & 12 & 59 
    \\
    \bottomrule
  \end{tabular}

  \end{table}
Nevertheless, the next theorem shows that the case of genus $5$ looks
also more interesting.

\begin{theorem}
  There exist at least three combinatorially distinct triangulations
  of a surface of genus 5 with 12 vertices that do not admit a
  polyhedral realization in $\R^3$. However, there exists at
  least one triangulation of a surface of genus 5 with 12 vertices
  that admits many oriented matroids.
\end{theorem}

Specifically no admissible oriented matroids exist for the manifolds
$\lutzm{2}{12}{1}{1}$, $\lutzm{2}{12}{1}{2}$, and
$\lutzm{2}{12}{1}{6}$ described in the dissertation of Frank Lutz
\cite{lutz:diss}. However, more than 100\,000 admissible oriented
matroids exist for the manifold $\lutzm{2}{12}{5}{1}$. A facet
description of the non-realizable manifolds can be found in the
Tables~\ref{tab:g51t}, \ref{tab:g52t}, \ref{tab:g55t}.

\begin{table}[tbp]
  \caption{Triangulation \lutzm{2}{12}{1}{1} of \ocite{lutz:diss}}
  \label{tab:g51t}
  \centering
  {\surff
  \begin{tabular}{rrr@{\hem}rrr@{\hem}rrr@{\hem}rrr@{\hem}rrr@{\hem}rrr}
  \toprule
  1&2&3&1&2&12&1&3&6&1&4&8&1&4&11&1&5&9\\
  1&5&10&1&6&9&1&8&10&1&11&12&2&3&4&2&4&7\\
  2&5&9&2&5&12&2&6&10&2&6&11&2&7&10&2&9&11\\
  3&4&5&3&5&8&3&6&10&3&7&11&3&7&12&3&8&11\\
  3&10&12&4&5&6&4&6&9&4&7&11&4&8&12&4&9&12\\
  5&6&7&5&7&10&5&8&12&6&7&8&6&8&11&7&8&9\\
  7&9&12&8&9&10&9&10&11&10&11&12\\
  \bottomrule
\end{tabular}}
\end{table}

\begin{table}[tbp]
  \caption{Triangulation \lutzm{2}{12}{1}{2} of \ocite{lutz:diss}}
  \label{tab:g52t}
  \centering
  {\surff
  \begin{tabular}{rrr@{\hem}rrr@{\hem}rrr@{\hem}rrr@{\hem}rrr@{\hem}rrr}
    \toprule
    1&2&3&1&2&12&1&3&6&1&4&9&1&4&11&1&5&8\\
    1&5&9&1&6&10&1&8&10&1&11&12&2&3&4&2&4&7\\
    2&5&10&2&5&12&2&6&9&2&6&10&2&7&11&2&9&11\\
    3&4&5&3&5&8&3&6&11&3&7&10&3&7&11&3&8&12\\
    3&10&12&4&5&6&4&6&9&4&7&12&4&8&11&4&8&12\\
    5&6&7&5&7&10&5&9&12&6&7&8&6&8&11&7&8&9\\
    7&9&12&8&9&10&9&10&11&10&11&12\\
    \bottomrule
  \end{tabular}}
\end{table}

\begin{table}[tbp]
  \caption{Triangulation \lutzm{2}{12}{1}{6} of \ocite{lutz:diss}}
  \label{tab:g55t}
  \centering
  {\surff
  \begin{tabular}{rrr@{\hem}rrr@{\hem}rrr@{\hem}rrr@{\hem}rrr@{\hem}rrr}
    \toprule
    1&2&4&1&2&6&1&3&6&1&3&12&1&4&11&1&5&9\\
    1&5&12&1&8&9&1&8&10&1&10&11&2&3&5&2&3&7\\
    2&4&7&2&5&12&2&6&10&2&9&10&2&9&11&2&11&12\\
    3&4&6&3&4&8&3&5&8&3&7&11&3&10&11&3&10&12\\
    4&5&7&4&5&9&4&6&9&4&8&12&4&11&12&5&6&8\\
    5&6&10&5&7&10&6&7&9&6&7&11&6&8&11&7&8&10\\
    7&8&12&7&9&12&8&9&11&9&10&12\\
    \bottomrule
  \end{tabular}}
\end{table}

\begin{table}[tbp]
  \centering
  \caption{Surface no.1 of \ocite{altshuler.ea:neighborly}}
  {\surff
  \begin{tabular}{rrr@{\hem}rrr@{\hem}rrr@{\hem}rrr@{\hem}rrr@{\hem}rrr}
    \toprule
    1&2&11&1&2&12&1&3&4&1&3&10&1&4&9&1&5&6\\
    1&5&11&1&6&9&1&7&8&1&7&10&1&8&12&2&3&6\\
    2&3&8&2&4&10&2&4&12&2&5&9&2&5&10&2&6&7\\
    2&7&11&2&8&9&3&4&6&3&5&7&3&5&12&3&7&8\\
    3&9&10&3&9&11&3&11&12&4&5&8&4&5&12&4&6&8\\
    4&7&10&4&7&11&4&9&11&5&6&10&5&7&9&5&8&11\\
    6&7&12&6&8&9&6&10&11&6&11&12&7&9&12&8&10&11\\
    8&10&12&9&10&12\\
    \bottomrule
  \end{tabular}}
  \label{tab:g6no1}
\end{table}

Another interesting question was dealt with by Bokowski and Guedes de
Oliveira \cite{bokowski.guedesdeoliveira:generation}: Are there
infinite classes of surfaces of a fixed genus that cannot be realized?
Bokowski and Guedes de Oliveira tried to answer this question by
taking a non-realizable surface and cutting out a triangle such that
the remaining manifold stays non-realizable. We found that the
remaining manifold given by Bokowski and Guedes de Oliveira admits
chirotopes after all. Their argument for non-realizability depends
crucially on the symmetry of the surface to reduce the search space.
However, the symmetry group of the manifold in question is smaller
than the symmetry group of the whole surface.

Still, our algorithm yields an even stronger statement:

\begin{theorem}\label{thm:infclass}
  For each genus $g\geq 5$ there exist infinite classes of surfaces that
  have no polyhedral embeddeding in $\R^3$.
\end{theorem}

The main idea to construct such an infinite family was already given
by Bokowski and Guedes de Oliveira
\cite{bokowski.guedesdeoliveira:generation}. We take the connected sum
of suitable surfaces; we can ensure that the result is non-realizable
if one of the summands stayed non-realizable after the removal of one
triangle. As we do not need to impose any conditions on the second
summand, we can then construct surfaces of arbitrary genus $g$ as long
as $g$ is greater or equal than the genus of the first summand; by
additionally adding triangulations of spheres with arbitrary numbers
of vertices we can construct the infinite families we are after. The
construction is summarized in the following Lemma. We omit the
straight-forward proof.

\begin{lemma}
  Given two triangulations $\asc{S}$ and $\asc{T}$ of surfaces and a
  triangle $T\in \asc{S}$ such that $\asc{S}\setminus\{T\}$ is
  non-realizable, then there exists a triangulation $\asc{X}$ with
  $V(\asc{X})=V(\asc{S})+V(\asc{T}) - 3$ vertices of the surface of
  genus $g_{\asc{S}}+g_{\asc{T}}$ that is non-realizable as well.
\end{lemma}

The following proposition shows that the conditions of the Lemma can
be satisfied. 

\begin{proposition}
  \begin{enumerate}[a)]
  \item Let $\asc{O}$ be the surface $\lutzm{2}{12}{1}{1}$ as above
    and let $\asc{M}:=\asc{O}\setminus\{\{1,2,3\}\}$. Then $\asc{M}$
    does not admit an acyclic uniform oriented matroid.
  \item Let $\asc{P}$ be the surface no.~1 in the enumeration of
    Altshuler and Bokowski \cite{altshuler.ea:neighborly} (see
    Table~\ref{tab:g6no1}) and let
    $\asc{N}:=\asc{P}\setminus\{\{1,2,11\}\}$. Then $\asc{N}$ does not
    admit an acyclic uniform oriented matroid.
\end{enumerate}
\end{proposition}

\begin{proof}[Proof of Theorem~\ref{thm:infclass}]
  As a first step we will show a construction that yields for any
  surface $\asc{X}$ a non-realizable surface $\asc{S}$. We then exhibit
  suitable sequences of surfaces to show the Theorem.
  
  Take a triangulated surface $\asc{X}$ of genus $g$ with $n$
  vertices.  After renumbering we may assume that the vertices are
  $13,\dotsc n+12$ and that $[n+10,n+11,n+12]$ is a triangle in $X$.
  Now we take the connected sum of $\asc{X}$ and $\asc{O}$ where we
  identify the pairs of vertices $(1,n+10)$, $(2,n+11)$, and
  $(3,n+12)$. We call this complex $\asc{S}$. It follows from the
  construction that $\asc{S}$ is a surface.  Furthermore, $\asc{S}$
  has genus $g+5$ and $n+9$ vertices. We claim that $\asc{S}$ cannot
  be realizable: it contains $\asc{M}$ as a subcomplex. As we have
  seen $\asc{M}$ is not realizable, so the claim follows.
  
  Now, let $g\geq 5$. Then let $\asc{X_0}$ by any triangulated surface
  of genus $g-5$ and let $\asc{X_i}$ be the connected sum of
  $\asc{X_0}$ with a triangulated sphere with $i+3$ points. We see
  that the sequence $\asc{S_0},\asc{S_1},\dotsc$ constructed as above
  is an infinite sequence of surfaces of genus $g$ all of which are
  not realizable.
\end{proof}

Our results depend on the following method to generate oriented
matroids. We first give an overview before we deal with the technical
details.

\section{The Main Algorithm}
\label{sec:algorithm}

We want to treat the embeddability problem algorithmically. To do so,
we need a combinatorial model of a point set in $\R^n$, which captures
interesting properties (for instance, convexity). Oriented matroids
are a good choice for this purpose. Examples of such applications can
be found for instance in the book by Bokowski and Sturmfels
\cites{bokowski.sturmfels:synthetic} .

In the realizable case the circuits of an oriented matroid correspond
to minimal Radon partitions of the corresponding elements.  We can use
this correspondence to check whether two simplices intersect each
other. If $F$ and $G$ are simplices such that $F\cap G=\emptyset$,
they intersect if and only if $F\cup G$ contains a circuit $C$ such
that $C^+\subseteq F$ and $C^-\subseteq G$. We say that an oriented
matroid $\mathcal{M}=(E,\chi)$ is \emph{admissible for a simplicial
  complex $\asc{K}$} if $E=|\asc{K}|$ and for
all $F,G\in\asc{K}$ with
$F\cap G=\emptyset$ there does not exist any circuit $C$ such that
$C^+\subseteq F$ and $C^-\subseteq G$. If we consider only uniform
oriented matroids of rank $4$ and our simplices are faces of a surface,
we only need to consider the case that $F$ is a triangle and $G$ is an
edge. Additionally, we use a known fact about oriented matroids that
are derived from point sets: no circuit of such an oriented matroid is
totally positive.  Oriented matroids with this property are called
\emph{acyclic}.

We can restrict our problem even further: polyhedral embeddings of
triangulated surfaces are ``nice''; we can perturb the vertices by a
small amount without creating any intersections of the triangles. This
makes our task of finding oriented matroids comparatively easy. We can
restrict our attention to \emph{uniform} oriented matroids.

So, for a given simplicial complex, we can deduce that $\asc{K}$
cannot be embedded in $\R^{d}$, if $\asc{K}$ does not admit any
acyclic, uniform oriented matroid of rank $d+1$. We will now check for
this condition by transforming it into an instance of SAT. Luckily,
this transformation is quite straightforward. However, we first review
some oriented matroid terminology and fix the notation for instances
of SAT. The main part consists of the encoding the oriented matroid
axioms, i.e.  the three-term Grassmann-Plücker relations, and of
encoding the ``forbidden'' circuits.

\subsection{Simplicial Complexes}
\label{sec:geom-simpl-compl}

We now give a rough sketch how oriented matroids can be used to tackle
realizability questions. Assume we have a realization of a
triangulated surface $\asc{S}$, i.e. a map $f:\asc{S}\rightarrow\R^3$
such that for all $\Delta_1,\Delta_2\in\asc{S}$ holds that
$\conv(f(\Delta_1\cap\Delta_2))=f(\conv(\Delta_1))\cap
f(\conv(\Delta_2))$. If we want that $f$ is an embedding, we need to
make sure that the image of two simplices has non-trivial intersection
if and only if the simplices themself intersected non-trivially. 

\begin{definition}[Embedding]
  Given a triangulation $\asc{K}$ of a surface, we say that a mapping
  $f:\asc{K}\rightarrow \R^d$ induces an embedding if for no two
  simplices that are disjoint in $\asc{K}$ their images under $f$
  intersect in $\R^d$.
\end{definition}

When we want to check whether a mapping is an embedding, we can
restrict our attention to simplices whose dimension sum to $d$.  In
our case this means we only need to check intersections of one
triangle with an edge that is disjoint from the triangle.

\subsection{Oriented Matroids}
\label{sec:oriented-matroids}

The following discussion of oriented matroids is extremely brief, we
recommend the monograph \cite{OM}, especially Section~3.5 for
the missing details.

We only consider uniform oriented matroids and assume these are given
by their chirotopes.  We also assume that the ground set $E$ of the
oriented matroids is $\{1,\dotsc,n\}$.  We use the following axioms for
oriented matroids:

\begin{definition}
  Let $E=\{1,\dotsc,n\}$, $r\in\mathbb{N}$, and $\chi: E^r \rightarrow
  \{-1,+1\}$. We call $\mathcal{M}=(E,\chi)$ a uniform oriented
  matroid of rank $r$, if the following conditions are satisfied:
  \begin{enumerate}[(B1)]
  \item The mapping $\chi$ is alternating.
  \item For all $\sigma\in\binom{n}{r-2}$ and all subsets
    $\{x_1,\dotsc,x_4\}\subseteq E\setminus\sigma$ the following
    holds:
    \begin{multline*}
      \{\chi(\sigma,x_1,x_2)\chi(\sigma,x_3,x_4),-\chi(\sigma,x_1,x_3)\chi(\sigma,x_2,x_4),\\\chi(\sigma,x_1,x_4)\chi(\sigma,x_2,x_3)\}\supseteq
      \{-1,+1\}
\end{multline*}
  \end{enumerate}
\end{definition}
\begin{remark}
  The mapping $\chi$ is called the \emph{chirotope} of the oriented
  matroid.
\end{remark}

As a first consequence of these axioms we can restrict our attention
to the values that $\chi$ attains on the ordered $r$-subsets of $E$.
The other values are then determined by (B1).

The class of oriented matroids we are interested in is still smaller
than the class of uniform oriented matroids. We also want our oriented
matroids to be \emph{acyclic}, that means the should contain no
circuit in which every element has positive signature. Oriented
matroids with a positive circuit are called \emph{cyclic}.

Given a uniform oriented matroid $\mathcal{M}=(E,\chi)$ the circuit
signatures of $\mathcal{M}$ can be computed from the chirotope: Let
$\underline{C}=[c_1,\dotsc,c_{r+1}]$ ($c_1 <\dotsb < c_{r+1}$) be the
unoriented circuit, then the two possible signatures $C^+$ and $C^-$
of $\underline{C}$ are given by
$C_i=(-1)^i\chi[c_1,\dotsc,\widehat{c_i},\dotsc,c_{r+1}]$ and its
negative (for a proof, see \cite{OM}*{Lemma~3.57}). 
Recalling the discussion in the section above, the circuit
signatures give us the possibility to check whether two simplices of
complementary dimensions intersect.

\subsection{SAT}
\label{sec:sat}

Before we give our transformation, we first fix our notation for
instances of SAT.

Take a Boolean function $\Phi: \{0,1\}^n\rightarrow \{0,1\}$, where
$0$ stands for false and $1$ for true. We call the elements of
$\{0,1\}^n$ \emph{valuations}. A valuation is \emph{satisfying} if
$\Phi(v)=1$.

We transform our problem, whether there exists an admissible oriented
matroid for a given simplicial complex, into an instance of SAT.

An instance of SAT consists of a boolean function given in conjunctive
normal form (CNF). That is, given the variables $p_1,\dotsc,p_n$ the
function $\Phi$ is of the form $\Phi(p)=\bigwedge_{i=0}^m C_i$ where
the $C_i$ are of the form $C_i=\bigvee_{j\in I_i} p_j \vee
\bigvee_{j\in \overline{I}_i} \overline{p}_j$. A SAT solver answers
the question whether $\Phi$ is satisfiable. In that case it returns a
valuation $v$ such that $\Phi(v)=1$.

The following observation goes back to Peirce \cite{peirce}. It gives
us a way to write an arbitrary boolean function in CNF.

\begin{lemma}\label{thm:peirce}
  Let $\Phi$ be a boolean function $\Phi:
  \{0,1\}^n\rightarrow\{0,1\}$. Then we can write $\Phi$ as:
  \[ \Phi(x)=\bigwedge_{\substack{\sigma\in\{0,1\}^n\\\neg\Phi(\sigma)}}
  \left(\bigvee_{i\in\{j\mid\sigma_j=1\}} \overline{x_i}\right) \vee \left(\bigvee_{i\in\{j\mid\sigma_j=0\}} x_i\right)\]
\end{lemma}

\subsection{Encoding}
\label{sec:encoding}

We are now ready to give the transformation of our problem: Given a
simplicial complex $\asc{K}$ on $n$ points and a dimension $d$, we want
to decide whether there exists an acyclic, uniform oriented matroid of
rank $d+1$ on $n$ points that is admissible for $\asc{K}$.

To encode the chirotope we introduce a variable for each ordered
$r$-subset $B$ of $\{1,\dotsc,n\}$ which we denote by $[B]$. Given a
valuation $v$ we construct a chirotope $\chi_v$ as follows: If
$v_{[B]}=1$, we set $\chi_v(B)=+1$ and if $v_{[B]}=0$ then we set
$\chi_v(B)=-1$.

We start by encoding the oriented matroid axioms. We do not deal
explicitly with the axiom (B1) as we only fix the signs for the
ordered subsets.  The following proposition allows us to deal with
axiom (B2). It follows directly from Lemma~\ref{thm:peirce}.

\begin{table}[htbp]
  \centering
{\begin{align*}
      GP(\alpha,\beta,\gamma,\delta,\epsilon,\zeta)&=  \left(\neg[\alpha]\vee\neg[\beta]\vee\neg[\gamma]\vee[\delta]\vee\neg[\epsilon]\vee\neg[\zeta]\right)\\
      &\wedge \left(\neg[\alpha]\vee\neg[\beta]\vee\neg[\gamma]\vee[\delta]\vee[\epsilon]\vee[\zeta]\right)\\
      &\wedge \left(\neg[\alpha]\vee\neg[\beta]\vee[\gamma]\vee\neg[\delta]\vee\neg[\epsilon]\vee\neg[\zeta]\right)\\
      &\wedge \left(\neg[\alpha]\vee\neg[\beta]\vee[\gamma]\vee\neg[\delta]\vee[\epsilon]\vee[\zeta]\right)\\
      &\wedge \left(\neg[\alpha]\vee[\beta]\vee\neg[\gamma]\vee\neg[\delta]\vee\neg[\epsilon]\vee[\zeta]\right)\\
      &\wedge \left(\neg[\alpha]\vee[\beta]\vee\neg[\gamma]\vee\neg[\delta]\vee[\epsilon]\vee\neg[\zeta]\right)\\
      &\wedge \left(\neg[\alpha]\vee[\beta]\vee[\gamma]\vee[\delta]\vee\neg[\epsilon]\vee[\zeta]\right)\\
      &\wedge \left(\neg[\alpha]\vee[\beta]\vee[\gamma]\vee[\delta]\vee[\epsilon]\vee\neg[\zeta]\right)\\
      &\wedge \left([\alpha]\vee\neg[\beta]\vee\neg[\gamma]\vee\neg[\delta]\vee\neg[\epsilon]\vee[\zeta]\right)\\
      &\wedge \left([\alpha]\vee\neg[\beta]\vee\neg[\gamma]\vee\neg[\delta]\vee[\epsilon]\vee\neg[\zeta]\right)\\
      &\wedge \left([\alpha]\vee\neg[\beta]\vee[\gamma]\vee[\delta]\vee\neg[\epsilon]\vee[\zeta]\right)\\
      &\wedge \left([\alpha]\vee\neg[\beta]\vee[\gamma]\vee[\delta]\vee[\epsilon]\vee\neg[\zeta]\right)\\
      &\wedge \left([\alpha]\vee[\beta]\vee\neg[\gamma]\vee[\delta]\vee\neg[\epsilon]\vee\neg[\zeta]\right)\\
      &\wedge \left([\alpha]\vee[\beta]\vee\neg[\gamma]\vee[\delta]\vee[\epsilon]\vee[\zeta]\right)\\
      &\wedge \left([\alpha]\vee[\beta]\vee[\gamma]\vee\neg[\delta]\vee\neg[\epsilon]\vee\neg[\zeta]\right)\\
      &\wedge
      \left([\alpha]\vee[\beta]\vee[\gamma]\vee\neg[\delta]\vee[\epsilon]\vee[\zeta]\right)
\end{align*}}
  \caption{Definition of the function $GP$}
  \label{tab:gp}
\end{table}

\begin{proposition}
  Let $\alpha,\beta,\gamma,\delta,\epsilon,\zeta$ be ordered
  $r$-subsets of $E$, $v\in\{0,1\}^{\binom{\lvert E\rvert}{r}}$, and
  $\chi_v$ defined as above. Then the following two conditions are
  equivalent:
  \begin{enumerate}
  \item    $\{\chi_v(\alpha)\chi_v(\beta),-\chi_v(\gamma)\chi_v(\delta),\chi_v(\epsilon)\chi_v(\zeta)\}\supseteq
    \{+1,-1\}$
\item $v$ satisfies $GP(\alpha,\beta,\gamma,\delta,\epsilon,\zeta)$ as
  defined in Table~\ref{tab:gp}. 
  \end{enumerate}
\end{proposition}


So, three-term Grassmann-Plücker relation is encoded with $16$ clauses
with $6$ literals each. As we have $\binom{n}{r-2}\binom{n-r+2}{4}$
different Grassmann-Plücker relations to consider, we get
$16\binom{n}{r-2}\binom{n-r+2}{4}$ many clauses of length $6$ in our
resulting SAT instance. These clauses guarantee the property that each
satisfiable valuation of the instance will correspond to a chirotope.

To complete the model we need a condition that excludes all oriented
matroids that have a given circuit signature. As a special case we
want to exclude all cyclic oriented matroids. The following
proposition gives the exact condition; again the proposition follows
directly from Lemma~\ref{thm:peirce}:

\begin{proposition}
  Let $\mathcal{M}=(E,\chi)$ be a uniform oriented matroid of rank
  $r$, $v$ the corresponding valuation and $C=(s_1 c_1,\dotsc,s_{r+1}
  c_{r+1})$ be a signed $(r+1)$-tuple ($s_i\in\{+1,-1\}$,
  $c_i\in\{1,\dotsc,n\}$, $c_i\neq a_j$). Then $C$ is \emph{not} a
  circuit of $\chi$ if and only if $v$ satisfies $\Gamma(C)$:
  \begin{multline*}
    \Gamma(C)= \bigwedge_{i\in
    I^+}[c_1,\dotsc,\widehat{c_i},\dotsc,c_{r+1}] \wedge
    \bigwedge_{i\in I^-} \neg
    {[c_1,\dotsc,\widehat{c_i},\dotsc,c_{r+1}]}\\
    \vee \bigwedge_{i\in
    I^+}\neg[c_1,\dotsc,\widehat{c_i},\dotsc,c_{r+1}] \wedge
    \bigwedge_{i\in I^-}
    {[c_1,\dotsc,\widehat{c_i},\dotsc,c_{r+1}]}\\
    I^+=\{i\mid (-1)^i s_i=+1\}\\
    I^-=\{i\mid (-1)^i s_i=-1\}\\
  \end{multline*}
\end{proposition}

Thus, we add for every forbidden circuit two clauses consisting of
$r+1$ literals each. With these clauses we have completed our
SAT-model. In the next section we will see how this gives us an
effective way to solve our problem. If we want to use this method to
treat other realizability problems, other restrictions are of
interest. In the case of the algorithmic Steinitz problem, i.e.
whether a lattice is a face lattice of a convex polytope, we need to
generate oriented matroids with prescribed cocircuits. The necessary
clauses can be derived in the same manner as described above.

\section{Implementation}
\label{sec:implementation}

We wrote a Haskell \cite{haskell} program that does the translation
described in the preceding section. We then used the SAT-solvers
ZChaff \cite{Zchaff} and Minisat \cite{minisat} to solve the resulting
SAT instances. To verify the data entry of the 59 Altshuler examples
we checked the resulting surfaces using Polymake \cite{polymake}.

We tested our programs on known examples. We computed all chirotopes
that are admissible for the Möbius torus. We found 2772 chirotopes (in
less than 20 seconds) which is the same number that Bokowski and
Eggert \cite{bokowski.eggert:moebius} found. Furthermore, we tested
all triangulated surfaces with up to 9 vertices (including the
non-orientable ones). In that case our program correctly found out
which surfaces (all the orientable ones) admitted chirotopes and which
did not.

Additionally, we used our program to verify that all $821$ minimal
vertex triangulations of a surface of genus $4$ as classified by Lutz
and Sulanke \cite{lutz.sulanke:isomorphism} admit a chirotope.

There are quite a number of software packages to generate oriented
matroids (for instance
\cites{bokowski.guedesdeoliveira:generation,dissgugisch,mpc,MR1871693}).
These packages use one of two different approaches: The programs by
Bokowski and Guedes de Oliveira and by Finschi construct oriented
matroids by using single element extensions, whereas the other
programs try to construct the oriented matroids globally by filling in
the chirotopes. Our approach is of the second type. We want to mention
that David Bremner reported on his transformation of the problem into
a 0/1 integer program.  However, his benchmark results showed that his
backtracking program is faster in the instances he used.

To give an impression of the efficiency of our program, we state some
running times: For the genus $6$ examples the transformation took
approximately 30 seconds per instance. Solving the SAT instances took
between 22 and 98 minutes. All times were taken on a machine with two
Pentium~III processors (1 GHz) and 2 GB RAM. For all computations only
one processor was used. These results show that our program is much
faster than the program of
\cite{bokowski.guedesdeoliveira:generation}. We think the most
interesting comparison would be to the program MPC of David Bremner.
However, he does not implement the possibility to exclude oriented
matroids with certain circuits.

One of the advantages of our method lies in the fact that one can use
a variety of SAT solvers to check the results. The transformation is
simple enough to be checked by hand. Many SAT solvers allow the
possibility to give a ``proof'' that an instance is unsatisfiable.
They output how to derive a contradiction from the given input.
However, this does not improve our situation: the proofs generated
this way are so large that they can only be checked with the help
of a computer. Advances in the development of proof assistants might
make it possible to give a full formal verification of our results in
the near future.

\section{Immersions}
\label{sec:immersions}

We have seen that we cannot hope to find embeddings for all
triangulations of orientable surfaces. However, one could hope for
weaker results. In the context of non-orientable surfaces, where
embeddings cannot be found for topological reasons, one tries instead
to find immersions of these surfaces. Thus, we could hope to find some
immersions for the surfaces we found not to be embeddable. 

We mention that Cervone \cite{cervone:immersions} showed there are
non-immersable triangulations with eight vertices of the Klein bottle,
whereas one can find an immersion of a triangulation with nine
vertices. Brehm had earlier shown that there is no gap between the the
necessary vertex numbers for immersions of the real projective plane
\cite{MR1076535}.

\begin{definition}[Immersion]
  Given a triangulation $\asc{K}$ of a surface, we say that a mapping
  $f:\asc{K}\rightarrow \R^d$ induces an immersion if for no two
  triangles in the star of a vertex $v\in\asc{K}$ their images under
  $f$ intersect in $\R^d$.
\end{definition}

\begin{remark}
  The star of a vertex is the smallest simplicial complex, that
  contains all faces that contain the given vertex.
\end{remark}

This definition directly leads to an adaptation of the notion of an
admissible oriented matroid. We say that an oriented matroid
$\mathcal{M}=(E,\chi)$ is \emph{admissible with respect to an
  immersion of a simplicial complex $\asc{K}$} if the following
conditions hold: 
\begin{itemize}
\item $E=|\asc{K}|$,
\item for all $F,G\in\str(v)$ with $F\cap G=\emptyset$ there does not
  exist any circuit $C$ such that $C^+\subseteq F$ and $C^-\subseteq
  G$.
\end{itemize}
Using a suitably modified version of the algorithm above (one just
needs to test fewer possible intersections), we can show that all but
one of the $59$ surfaces of genus $6$ do not admit an oriented
matroid that is admissible with respect to an immersion of that
surface. The exception is surface number $15$ (again using the
numbering scheme used by Altshuler et
al. \cite{altshuler.ea:neighborly}). 

\section{Conclusion}
\label{sec:conclusion}

Our results give additional insight in the properties of minimal
vertex triangulations of surfaces. Still, the main problems remain:
How can we characterize non-realizability? Are all triangulated
surfaces of small genus (i.e. $g\leq 4$) realizable?

The infinite class of non-realizable surfaces given above hints that
there will be no easy answer to the first question. For genus $5$ and
$6$ we can construct non-realizable triangulations for any number of
vertices. We conjecture that this holds also for any genus larger than
$6$. However, we think it should be possible to prove that for every
genus greater than $4$ we need strictly more vertices for a polyhedral
embedding of a surface than for a combinatorial triangulation.
However, one of the main obstacles for such an investigation is the
lack of good construction methods for ``interesting'' combinatorial
surfaces.

The method we used is interesting in its own right. It helps
tremendously in the study of small examples. However, we hope that the
small examples given here will help in the solution of the general
problem. One point that needs improvement is the fact that we cannot
use effectively use the information we gain if we find oriented
matroids in the course of our search. The methods for finding
realizations of oriented matroids are not good enough to yield
practical results.

As an open problem remains the question how strong oriented matroid
methods are compared to the methods described by Novik
\cite{novik:note} and Timmreck~\cite{timmreck:necessary}.  We
conjecture that using oriented matroids will give as strong results as
the method proposed by Timmreck. We are lead to this conjecture by the
result of Carvalho and Guedes de Oliveira
\cite{carvalho.guedesdeoliveira:intersection}. They showed that the
linking number arguments given by Brehm as incorporated by Timmreck
hold also in the setting of oriented matroids. That means that these
arguments are subsumed by the oriented matroid technique.

The technique used in this article can be applied to other geometric
problems. It has already been used to treat realizability of
point-line configurations. Another application could be in tackling
the Algorithmic Steinitz problem (cf. \cite{OM}). We hope that this
technique proves itself to be a useful building block in these and
other applications.

 \section*{Acknowledgements}
 I would like to thank Jürgen Bokowski for all the advice and
 encouragement received while undertaking this work. I had a number of
 helpful discussions with a number of people: I like to mention Peter
 Lietz, Frank Lutz (who also generously provided me lots of
 interesting examples for testing my software), and Dagmar Timmreck.
 Thanks also to Michael Joswig who gave a number of valuable comments
 on an earlier draft of this article.

\end{document}